\documentclass[a4paper,10pt,reqno]{amsart}
\usepackage{fullpage}
\usepackage{hyperref}
\usepackage{amsmath}
\usepackage{amsfonts}
\usepackage{amssymb}
\usepackage{amsthm}
\usepackage{latexsym}
\usepackage{mathtools}

\providecommand\given{}
\newcommand\SetSymbol[1][]{%
	\nonscript\:#1\vert
	\allowbreak
	\nonscript\:
	\mathopen{}}
\DeclarePairedDelimiterX\Set[1]\{\}{%
	\renewcommand\given{\SetSymbol[\delimsize]}
	#1
}

\DeclarePairedDelimiter{\abs}{\lvert}{\rvert}
\DeclarePairedDelimiterXPP{\norm}[2]{}{\lVert}{\rVert}{_{#2}}{#1}
\DeclarePairedDelimiterXPP{\Altnorm}[3]{}{\lVert}{\rVert}{_{#2}^{#3}}{#1}

\DeclareMathOperator{\WF}{WF}

\DeclareMathOperator{\Char}{Char}

\newcommand{\fa}{\;\,\forall\,}
\newcommand{\ex}{\;\,\exists\,}

\newcommand{\N}{\mathbb{N}}
\newcommand{\Real}{\mathbb{R}}
\newcommand{\Comp}{\mathbb{C}}
\newcommand{\E}{\mathcal{E}}
\newcommand{\D}{\mathcal{D}}
\newcommand{\CC}{\mathcal{C}}
\newcommand{\An}{\mathcal{A}}
\newcommand{\G}{\mathcal{G}}

\newcommand{\bG}{\mathbf{G}}
\newcommand{\sP}{\mathsf{P}}

\newcommand{\alp}{{\lvert\alpha\rvert}}
\newcommand{\bet}{{\lvert\beta\rvert}}

\newcommand{\lamb}{{\lvert\lambda\rvert}}
\newcommand{\gam}{{\lvert\gamma\rvert}}

\newcommand{\Beu}[2]{\mathcal{E}^{( #1 )} ( #2 )}
\newcommand{\Rou}[2]{\mathcal{E}^{\{ #1 \}} \left( #2 \right)}
\newcommand{\DC}[2]{\mathcal{E}^{[ #1 ]} ( #2 )}

\newcommand{\gRou}[2]{\mathcal{B}^{\{ #1 \}} ( #2 )}
\newcommand{\gBeu}[2]{\mathcal{B}^{( #1 )} ( #2 )}

\newcommand{\gDC}[2]{\mathcal{B}^{[ #1 ]} ( #2 )}

\newcommand{\vBeu}[3][P]{\mathcal{E}^{ (#2) } ( #3 ; #1 )}
\newcommand{\vRou}[3][P]{\mathcal{E}^{\{ #2 \}} \left( #3; #1 \right)}
\newcommand{\vDC}[3][P]{\mathcal{E}^{[ #2 ]} ( #3 ; #1  )}

\newcommand{\vgBeu}[3][P]{\mathcal{B}^{ (#2) } ( #3 ; #1 )}
\newcommand{\vgRou}[3][P]{\mathcal{B}^{\{ #2 \}} \left( #3; #1 \right)}
\newcommand{\vgDC}[3][P]{\mathcal{B}^{[ #2 ]} ( #3 ; #1  )}

\newcommand{\Dp}{\mathcal{D}^\prime}

\newcommand{\Ban}[3][h]{\mathcal{B}_{ #2 , #1 }(#3)}

\newcommand{\bM}{\mathbf{M}}
\newcommand{\bN}{\mathbf{N}}
\newcommand{\fM}{\mathfrak{M}}
\newcommand{\fN}{\mathfrak{N}}
\newcommand{\bL}{\mathbf{L}}

\newcommand{\bB}{\mathbf{B}}

\newcommand{\fW}{\mathfrak{W}}
\newcommand{\bW}{\mathbf{W}}

\newcommand{\fT}{\mathfrak{T}}

\newcommand{\U}{\mathcal{U}}

\DeclareMathOperator*{\ind}{ind}
\DeclareMathOperator*{\proj}{proj}

\theoremstyle{plain}
\newtheorem{Thm}{Theorem}[section]
\newtheorem{Prop}[Thm]{Proposition}
\newtheorem{Lem}[Thm]{Lemma}
\newtheorem{Cor}[Thm]{Corollary}

\theoremstyle{definition}
\newtheorem{Def}[Thm]{Definition}

\newtheorem*{Conj}{Conjecture}

\theoremstyle{remark}
\newtheorem{Rem}[Thm]{Remark}
\newtheorem{Ex}[Thm]{Example}

\numberwithin{equation}{section}

\title{The Kotake-Narasimhan Theorem in general ultradifferentiable classes}
\author{Stefan F\"urd\"os}
\address{Instituto de Matem\'atica e Estat\'istica, Universidade de S\~{a}o Paulo, Rua do Mat\~{a}o 1010,\\ 05508-090 S\~{a}o Paulo, SP, Brazil}
\email{stefan.fuerdoes@univie.ac.at}
\subjclass[2020]{Primary 35B65; Secondary 26E10,35B45,35J48,46E10}
	\keywords{Kotake-Narasimhan Theorem; problem of iterates; ultradifferentiable functions; ultradifferentiable vectors}
\thanks{text}

\begin{document}
\maketitle
\begin{abstract}
	We prove a Kotake-Narasimhan type theorem in general ultradifferentiable classes given by weight matrices. In doing so we simultaneously recover and  generalize significantly the known results for classes given by weight sequences and weight functions. In particular, we obtain a sharp Kotake-Narasimhan theorem for Beurling classes.
\end{abstract}

\section{Introduction}\label{Intro}
\subsection{Historical Background}
In this paper we consider the problem of iterates for elliptic operators with coefficients in general ultradifferentiable structures.\footnote{Here we understand by an ultradifferentiable 
	structure an algebra of smooth functions defined by
	estimates on the derivatives and which contains the analytic class.}
In its general form the problem of iterates for an ultradifferentiable structure $\U$ can be stated
in the following way:
\begin{quote}
	Let $u$ be a smooth function which satisfies the defining estimates of $\U$ with respect to 
	the iterates of a differential operator $P$.
	Can we conclude that $u$ is already an element of $\U$?
\end{quote}
If the answer to the question above is affirmative for an ultradifferentiable structure $\U$ and an operator $P$ then we say that the theorem of iterates holds for the operator $P$ with respect
to $\U$.

 Recently, there has been a resurged interest in the problem of iterates
 in various different settings, see e.g.\ \cite{MR3652556},\cite{ChailiDjilali18}, \cite{MR4098643}, \cite{zbMATH07538293}  and
 \cite{zbMATH07469041}.
  For surveys on the problem of iterates we refer the reader to \cite{MR1037999} and \cite{Derridj2017}.
In this paper we are going to revisit the classical case of elliptic operators in open
sets of ${\Real}^n$ 
in view of the recently expanded theory on general ultradifferentiable classes.

The main starting point of the problem of iterates was when in 1962 Kotake-Narasimhan \cite{KotakeNarasimhan} and Komatsu \cite{komatsu1962}
separately proved the following statement:
 If $P$ is an elliptic operator of order $d$
with analytic coefficients on an open set $\Omega\subseteq{\Real}^n$, then a
smooth function $u\in\E(\Omega)$ is analytic in $\Omega$ if and only if
 for each relatively compact set $U\subseteq \Omega$ there are constants $C,h>0$ such that
\begin{equation*}
	\norm*{P^ku}{L^2(U)}\leq Ch^k(dk)!, \quad k\in\N_0.
\end{equation*}
Nelson \cite{MR107176} proved an analogous statement for an elliptic
system of analytic vector fields.

The next natural step is then to ask if a similar result holds if 
one considers instead of the analytic class more
general ultradifferentiable classes.
We are going to say that the Kotake-Narasimhan theorem holds for an ultradifferentiable structure $\U$
if the theorem of iterates with respect to $\U$ holds for every elliptic operator $P$ with coefficients in $\U$. 

For example Bolley-Camus \cite{MR548225} proved the Kotake-Narasimhan theorem for Gevrey classes.
If we want to consider more general families of ultradifferentiable classes then the commonly
used spaces are the 
Denjoy-Carleman classes which are defined by weight sequences, for an introduction
 see e.g.\ \cite{Komatsu73},
and the Braun--Meise--Taylor classes which are given by weight functions, introduced
in their modern form by \cite{MR1052587}.
Both of these classes are generalizations of the Gevrey classes,
however, they do not in general coincide, see \cite{BonetMeiseMelikhov07}.
In this paper we take a step further and consider ultradifferentiable classes
given by weight matrices, i.e.~countable families of weight sequences, which were introduced
in \cite{MR3285413} and \cite{MR3462072}. 
These classes include both Denjoy-Carleman and Braun-Meise-Taylor classes.
In \cite{zbMATH07538293} we showed that the theorem of iterates with respect to classes
given by weight matrices holds for elliptic operators with analytic coefficients using a microlocal
analytic approach, generalizing results of \cite{MR557524} in the case of weight sequences
and of \cite{MR3661157} for weight functions.

More generally, the Kotake-Narasimhan theorem for Denjoy-Carleman classes was proven in \cite{Lions}
and in a more general form in \cite{MR3090868}.
In the case of Braun-Meise-Taylor classes the Kotake-Narasimhan theorem was shown by
\cite{MR3652556}. All these instances followed generally the lines of the proofs
of \cite{komatsu1962}, \cite{KotakeNarasimhan} and \cite{MR548225}, which
 used a technique involving a-priori $L^2$-estimates  
and nested neighborhoods first introduced by Morrey-Nirenberg \cite{Morrey1957}.
In this paper we prove the Kotake-Narasimhan theorem for ultradifferentiable classes given by weight matrices by adapting the proof in \cite{MR548225} resp.\ \cite{ChailiDjilali18}.
In doing so we not only recover the known statements in Denjoy-Carleman classes and Braun-Meise-Taylor classes but also partially generalizing them.
In particular, in the case of Denjoy-Carleman classes Lions and Magenes \cite{Lions}
asked, what the optimal conditions on the weight sequences are in order for
 the Kotake-Narasimhan Theorem to hold.
As we will see, our main result implies especially that the Kotake-Narasimhan Theorem holds for the Denjoy-Carleman classes determined by
the sequences 
\begin{equation*}
	N_k^q=q^{k^2},\qquad k\in\N_0, 
\end{equation*}
where $q>1$ is a real parameter. These sequences have not been covered by the previous works on
the Kotake-Narasimhan Theorem for Denjoy-Carleman classes, cf.~Remark \ref{DC-Remark}.

\subsection{Statement of main results}
In order to formulate our main results for weight sequences and weight functions we need to fix
some notations: $\Omega$ will always be an open set of ${\Real}^n$ and we set $D_j=-i\partial_j=-i\partial_{x_j}$ where $\partial_{x_j}$ is
the $j$-th partial derivative, $j=1,\dotsc,n$.
We denote the set of positive integers by $\N$ whereas $\N_0=\N\cup\{0\}$.
Furthermore we say that a sequence $\bM=(M_k)_k$ of positive numbers
is a \emph{weight sequence} if $M_0=1\leq M_1$ and
\begin{equation}\label{WeakLogConvex}
	M_{k}^2\leq M_{k-1}M_{k+1}
\end{equation}
for all $k\in\N$.
Next we define the Denjoy-Carleman class associated to $\bM$ over an open set $\Omega\subseteq{\Real}^n$. 
To a weight sequence $\bM$ we can in fact associate two different ultradifferentiable classes.
First the Roumieu class associated to $\bM$ is given by
\begin{align*}
	\Rou{\bM}{\Omega}&=\biggl\{f\in\E(\Omega)\;\Big\vert \fa U\Subset \Omega\ex h>0 \ex C>0:\\ 
&\qquad\qquad\qquad\qquad\qquad		\sup_{x\in U}\,\abs*{D^\alpha f(x)}\leq Ch^\alp M_\alp\; \fa \alpha\in\N_0^n\biggr\}\\
	\shortintertext{whereas the Beurling class associated to $\bM$ is}
	\Beu{\bM}{\Omega}&=\biggl\{f\in\E(\Omega)\;\Big\vert \fa U\Subset \Omega \fa h>0 \ex C>0:\\
	&\qquad\qquad\qquad\qquad\qquad  
		\sup_{x\in U}\,\abs*{D^\alpha f(x)}\leq Ch^\alp M_\alp\; \fa \alpha\in\N_0^n\biggr\}.
\end{align*}
A basic question in the theory of ultradifferentiable classes is that of quasianalyticity.
We recall that an algebra $E$ of smooth functions is called non-quasianalytic if 
the only flat function in $E$ is the zero function.
 The Denjoy-Carleman theorem, see \cite{MR1996773}, says that
a Denjoy-Carleman class $\DC{\bM}{\Omega}$\footnote{We use the notation
	$[\ast]=\{\ast\},(\ast)$, $\ast$ denoting a weight sequence, weight function or weight matrix, in order to write down statements in the Roumieu and Beurling case simultaneously.} is non-quasianalytic if and only if
\begin{equation}\label{non-quasi}
	\sum_{k=1}^\infty \frac{M_{k-1}}{M_k}<\infty.
\end{equation}
We call a weight sequence $\bM$ non-quasianalytic if \eqref{non-quasi} holds and otherwise quasianalytic.


In order to formulate the Kotake-Narasimhan Theorem for Denjoy-Carleman classes
we have to specify additional conditions on the weight sequence $\bM$.
It is often easier to formulate these conditions not in terms of the sequence $(M_k)_k$ directly
but to use other sequences associated to $\bM$, like $m_k=M_k/k!$ or $\mu_k=M_k/M_{k-1}$.
\begin{Def}\label{SeqNormal}
	Let $\bM=(M_k)_k$ be a weight sequence.
We say that $\bM$ is \emph{weakly regular}
	if the following conditions hold:
		\begin{gather}
			\label{M0}				\lim_{k\rightarrow\infty} \sqrt[k]{m_k}=\infty,\\
			\label{M1} \text{the sequence }\sqrt[k]{m_k} \text{ is increasing.}\\
			\label{M2prime}
		\ex \gamma>0:\; M_{k+1}\leq \gamma^{k+1}M_k\quad \fa k\in\N_0.
		\end{gather}
\end{Def}
\begin{Ex}
	Let $s\geq 1$. The Gevrey sequence $\bG^{s}$, which is given by $G^s_k=(k!)^s$, is weakly regular.
	More generally the weight sequences $\bB^{s,\sigma}$, defined by
	$B^{s,\sigma}_k=(k!)^s\log(k+e)^{\sigma k}$, are weakly regular
	for all $s\geq 1$ and $\sigma >0$. The sequences $\bB^{s,\sigma}$ are quasianalytic
	if and only if $s=1$ and $0<\sigma\leq 1$.
	
	Another examples of weakly regular weight sequences are the following:
	For $q>1$ and $1<r$ let $\bL^{q,r}$ be given by $L_k^{q,r}=q^{k^r}$.
	The weight sequences $L_k^{q,r}$ are weakly regular for all $q>1$ and $1<r\leq 2$.
	In particular the case $r=2$ are the $q$-Gevrey sequences $\bN^q$ given by
	$N^q=q^{k^2}$.
\end{Ex}

If $\sP=\{P_1,\dotsc,P_\ell\}$ is a system of partial differential operators
\begin{equation*}
	P_j=\sum_{\alp\leq d_j} a_{j\alpha}D^\alpha
\end{equation*}
with smooth coefficients $a_{j\alpha}\in\E(\Omega)$, then  
we recall that the system $\sP$ is \emph{elliptic} in $\Omega$ if for every $x\in\Omega$
the principal symbols
\begin{equation*}
	p_j(x,\xi)=\sum_{\alp=d_j} a_{j\alpha}(x)\xi^\alpha, \qquad \xi\in{\Real}^n,
\end{equation*}
have no common nontrivial real zero in $\xi$.

Our main result in the case of Denjoy-Carleman classes is the following theorem.
	\begin{Thm}\label{Main-Weight}
		Let $\bM$ be a weakly regular weight sequence
		 and $\sP=\{P_1,\dotsc,P_\ell\}$ be
		an elliptic system of
		differential operators.
		 Then the following statements hold:
		 \begin{enumerate}
		 \item	Assume that the coefficients of the operators $P_j$, $j=1,\dotsc,\ell$ are
		 	all elements of $\Rou{\bM}{\Omega}$. 
		 	Then $u\in\Rou{\bM}{\Omega}$ if and only if $u\in\E(\Omega)$ and
		 	for all $U\Subset\Omega$ there are constants $C,h>0$ such that for all $k\in\N_0$ we have that
		 	\begin{equation*}
		 		\norm{P^\alpha u}{L^2(U)}\leq Ch^{d_\alpha}M_{d_\alpha}
		 	\end{equation*}
	 	for all  $\alpha\in\{1,\dotsc,\ell\}^k$ where $P^{\alpha}=P_{\alpha_1}\dots P_{\alpha_k}$ and $d_\alpha=d_{\alpha_1}+\dots+d_{\alpha_k}$ with $d_j$ denoting the order
	 	of the operator $P_j$, $j=1,\dotsc,\ell$.
	 	\item Assume that the coefficients of the operators $P_j$, $j=1,\dotsc,\ell$, 
	 	are in $\Beu{\bM}{\Omega}$.
	 	Then $u\in\Beu{\bM}{\Omega}$ if and only if $u\in\E(\Omega)$ and for all $U\Subset \Omega$
	 	 and every $h>0$  there exists 
	 	a constant $C>0$ such that
	 	\begin{equation*}
	 		\norm{P^\alpha u}{L^2(U)}\leq Ch^{d_\alpha}M_{d_\alpha}
	 	\end{equation*}
 	for all $k$ and $\alpha\in\{1,\dotsc,\ell\}^k$.
		 \end{enumerate}
	\end{Thm}
\begin{Rem}
	The notion of a weakly regular weight sequence is inspired by \cite{Dynkin80}.
	In that article a weight sequence $\bM$ is called regular if $\bM$ satisfies \eqref{M0}, \eqref{M2prime} and instead of \eqref{M1} the sequence $\bM$ is strongly logarithmic convex, i.e.
	\begin{equation}\label{StrLogConv}
		m_{k}^2\leq m_{k-1}m_{k+1}\qquad \fa k\in\N.
	\end{equation} 
In fact, \eqref{StrLogConv} implies \eqref{M1}.
\end{Rem}
\begin{Rem}\label{DC-Remark}
	We need to point out that
	Theorem \ref{Main-Weight} is a considerably more general statement then the previous known results,
	like \cite{Lions}, \cite{MR3090868} or \cite{ChailiDjilali18} in the Roumieu case
	and \cite{MR1875423} in the Beurling case.
	In these papers varying conditions on the weight sequence are assumed,
	but these conditions always include \eqref{non-quasi}, 
	\begin{gather}\label{StrLogConv2}
		m_{k}m_\ell\leq m_{k+\ell}\qquad \fa k,\ell\in\N_0\\
		\shortintertext{and}
		\label{M2}
		\ex \gamma>0:\; M_{j+k}\leq \gamma^{j+k+1} M_{j}M_k,\quad \fa j,k\in\N_0.
	\end{gather}
To compare these conditions with those of our result we may note that it is easy to see that
 \eqref{StrLogConv2} is another consequence of \eqref{StrLogConv}.
 However, if $\bM$ is a weight sequence we can see that \eqref{M1} implies
 also \eqref{StrLogConv2}: For $k=0$ we have that $m_0m_\ell=m_\ell$ for all $\ell\in\N_0$
 and the same is true for $\ell=0$ and all $k\in\N_0$.
 So we have to show $\eqref{StrLogConv2}$ for $k,\ell\in\N$, but then we have
 \begin{equation*}
 	m_km_\ell\leq m_{k+\ell}^{k/(k+\ell)}m_{k+\ell}^{\ell/(k+\ell)}=m_{k+\ell}
 \end{equation*}
by \eqref{M1}. Thence in that regard our conditions are formally more restrictive, but
this is compensated by the fact that we replaced
the other conditions noted above by far weaker conditions. 

   The most interesting property above in that regard is the last one.
It is clear that \eqref{M2} implies \eqref{M2prime}. 
However, \eqref{M2} is far more restrictive 
than \eqref{M2prime}, see e.g.~\cite{10.1215/kjm/1250520659}.
In fact, if a weight sequence $\bM$ satisfies \eqref{M2} then there is some $s>1$
such that $\DC{\bM}{\Omega}\subseteq\G^s(\Omega)$.
For example, the sequences $\bL^{q,r}$ cannot satisfy \eqref{M2} for any choice of $q,r>1$.

Moreover, we have replaced \eqref{non-quasi} by the  non-analyticity condition \eqref{M0}.
We will see in the next section, cf.~Remark \ref{StrongRemark}, that \eqref{M0}
is still nearly superfluous but we use it in order to allow for a unified formulation of 
Theorem \ref{Main-Weight}. This only excludes formally the analytic case, which is
well-known to hold.

We may also observe that in the Beurling case Theorem \ref{Main-Weight} is also a strict statement
in the sense that we allow that the coefficients are in the class given by the same weight sequence
as the space of vectors considered.
In contrast, for example \cite{MR1875423} requires that the coefficients of the operators
are in a strictly smaller class than the vectors considered.
We are able to remove this restriction by applying an argument given in \cite{MR550685}, which
essentially allows us to reduce the Beurling case to the Roumieu case.

\end{Rem}

In the case of Braun-Meise-Taylor classes our main theorem boils down to the following statement.
Recall that a weight function $\omega$ is an increasing continuous function $\omega: [0,\infty)\rightarrow [0,\infty)$ satisfying $\omega(t)=0$ for $t\in [0,1]$ and
\begin{gather}
\tag{$\alpha$}	\omega(2t)=O(\omega(t)),\qquad t\longrightarrow\infty,\\
\tag{$\beta$}	\omega(t)=o(\log t),\qquad\, t\longrightarrow\infty,\\
\tag{$\gamma$}	\varphi_\omega:=\omega\circ\exp\text{ is convex}.
\end{gather}
The conjugate function $\varphi^\ast_\omega(t) =\sup_{s\geq0} 
(st-\varphi_\omega(s))$, $t\geq 0$, is 
also convex, continuous and increasing and $\varphi^{\ast\ast}_\omega=\varphi_\omega$.
A smooth function $f\in\E(\Omega)$ is an element of the Roumieu class 
$\Rou{\omega}{\Omega}$ (resp.\ the Beurling class $\Beu{\omega}{\Omega}$) if
for any  $U\Subset\Omega$ there are constants $C,h>0$ (resp.\ for all $h>0$ there is
a constant $C=C_{h,U}$) such that
\begin{equation*}
\sup_{x\in U}\abs*{D^\alpha f(x)}\leq Ce^{\tfrac{1}{h}\varphi_\omega^\ast(h\alp)}\qquad 
\fa\alpha\in\N_0^n.
\end{equation*}
\begin{Thm}\label{MainTheoremOmega}
	Let $\omega$ be a concave weight function such that $\omega(t)=o(t)$
	and $\sP=\{P_1,\dotsc,P_\ell\}$ be an elliptic system of differential operators.
	\begin{enumerate}
		\item Assume that all coefficients of the operators $P_j$, $j=1,\dotsc,\ell$,
		are in $\Rou{\omega}{\Omega}$.
		Then $u\in\Rou{\omega}{\Omega}$ if and only if $u\in\E(\Omega)$ and for all 
		$U\Subset \Omega$ there are constants $C,h>0$ such that
		\begin{equation*}
			\norm*{P^\alpha u}{L^2(U)}\leq Ce^{\tfrac{1}{h}\varphi_\omega^\ast(hd_\alpha)}
		\end{equation*}
	for all $\alpha\in\{1,\dotsc,\ell\}^k$ and every $k\in\N_0$.
	\item Assume that all coefficients of the $P_j$, $j=1,\dotsc,\ell$ are
	in $\Beu{\omega}{\Omega}$. 
	Then $u\in\Beu{\omega}{\Omega}$ if and only if $u\in\E(\Omega)$ and for all $U\Subset\Omega$
	and every $h>0$ there is a constant $C>0$ such that
	\begin{equation*}
		\norm*{P^\alpha u}{L^2(U)}\leq Ce^{\tfrac{1}{h}\varphi_\omega^\ast(hd_\alpha)}
	\end{equation*}
	for all $\alpha\in\{1,\dotsc,\ell\}^k$ and every $k\in\N_0$.
	\end{enumerate}
\end{Thm}
\begin{Rem}
	Theorem \ref{MainTheoremOmega} generalizes the main theorem in \cite{MR3652556} from a single 
elliptic operator to elliptic systems of operators.
Furthermore, in the Beurling case
the Kotake-Narasimhan theorem 
was proven in \cite{MR3652556}, again,
 only for operators whose coefficients are in a strictly smaller class 
$\Beu{\tau}{\Omega}\subsetneq \Beu{\omega}{\Omega}$.
\end{Rem}
As we have announced above
we are going to prove our main Theorem \ref{localThmIterates}
in the far more general setting of weight matrices,
that is countable families of weight sequences.
We utilize the approach of \cite{MR548225}, which in particular allows us to apply the
aforementioned technique
of \cite{MR550685} 
in order to prove the Beurling version of Theorem \ref{localThmIterates} by effectively reducing
it to the Roumieu case even in the case of weight matrices.
However, we need to point out that in general it is not possible to use this technique in the  case
of non-trivial weight matrices, see for example \cite[Section 7]{FURDOS2020123451}.

The structure of the paper is the following: In Section \ref{Structure} we recall basic definitions
and facts about weight sequences, weight matrices and the ultradifferentiable classes of functions
and vectors generated by
them. Having sufficiently developed the theory of ultradifferentiable structures given by
weight matrices we can formulate our main result Theorem \ref{localThmIterates} at the end
of Section \ref{Structure}. In Section \ref{Sec:Estimates} we prove the fundamental $L^2$-estimate
which is used in Section \ref{Sec:Proof} to prove Theorem \ref{localThmIterates}.
We conclude the paper by stating some remarks in Section \ref{Remarks}.

\section{Ultradifferentiable structures}\label{Structure}
	\subsection{Weight sequences}

The following properties of a weight sequence are well known.

\begin{Lem}\label{weakLemma}
	Let $\bM$ be a weight sequence and set $\Lambda_k=\sqrt[k]{M_k}$
	and $\mu_k=\frac{M_k}{M_{k-1}}$ for $k\in\N$. Then
	the following holds:
	\begin{enumerate}
		\item $M_jM_k\leq M_{j+k}$ for all $j,k\in\N_0$.
		\item The sequence $\Lambda_k$ is increasing, i.e.\ $\Lambda_k\leq \Lambda_{k+1}$ for all $k\in\N$.
		\item The sequence $\mu_k$ is increasing.
		\item $\Lambda_k\leq \mu_k$ for all $k$.
	\end{enumerate}
\end{Lem}



For later convenience we shall take a closer look at the structure of Denjoy-Carleman classes,
for more details see \cite{Komatsu73}. Here we do not to require that $\bM$ is a weight sequence.

So let $U\subseteq{\Real}^n$ be an open set, $\bM$ be an increasing sequence of positive numbers and $h>0$ be a parameter.
Then $\E(\overline{U})$ is the space of smooth functions $f$ on $U$ such that $\partial^\alpha f$
extends continuously to $\overline{U}$ and 
 we define a seminorm on $\E(\overline{U})$ by setting
\begin{equation*}
	\norm*{f}{\bM,U,h}:=\sup_{\substack{x\in U\\ \alpha\in\N_0^n}}\frac{\abs*{D^\alpha f(x)}}{h^\alp M_\alp}.
\end{equation*}
Thus
\begin{equation*}
	\Ban{\bM}{U}=\Set*{f\in\E\bigl(\overline{U}\bigr)\given \norm*{f}{\bM,U,h} <\infty}
\end{equation*}
is a Banach space.
Then $\gRou{\bM}{U}=\ind_{h>0}\Ban{\bM}{U}=\bigcup_{h>0} \Ban{\bM}{U}$ resp.\
$\gBeu{\bM}{U}=\proj_{h>0}\Ban{\bM}{U}=\bigcap_{h>0}\Ban{\bM}{U}$ is the 
Roumieu resp.\ Beurling class of global ultradifferentiable functions associated to 
the weight sequence $\bM$ over $U$.

The local ultradifferentiable classes $\DC{\bM}{U}$ associated to $\bM$ over $U$ 
can thus be described as
$\DC{\bM}{U}=\proj_{V\Subset U}\gDC{\bM}{U}$.
If $\bM$ and $\bN$ are two sequences we write
\begin{align*}
	\bM\leq\bN\quad &:\Longleftrightarrow \quad \fa k\in\N_0: M_k\leq N_k,
	\\
	\bM\preceq\bN\quad &:\Longleftrightarrow \quad \left(\tfrac{M_k}{N_k}\right)^{1/k}
	\text{ is bounded for } k\rightarrow \infty,
	\\
	\bM\lhd\bN\quad &:\Longleftrightarrow \quad \left(\tfrac{M_k}{N_k}\right)^{1/k}\longrightarrow
	0 \text{ if } k\rightarrow \infty.
\end{align*}
If $\bM\preceq\bN$ then $\gDC{\bM}{U}\subseteq\gDC{\bN}{U}$ and $\DC{\bM}{U}\subseteq\DC{\bN}{U}$.
We are going to write $\bM\approx\bN$ if $\bM\preceq\bN$ and $\bN\preceq\bM$.
Furthermore $\bM\lhd\bN$ implies $\gRou{\bM}{U}\subseteq\gBeu{\bN}{U}$ and $\Rou{\bM}{U}\subseteq\Beu{\bN}{U}$.


We recall that to a weight sequence $\bM$ (or an abritary sequence) we associate another sequence $m_k=M_k/k!$.
It follows from above that the condition
\begin{equation}
	\lim_{k\rightarrow\infty}\sqrt[k]{m_k}=\infty\label{AnalyticInclusion1}
\end{equation}
implies that $\An(U)\subsetneq\DC{\bM}{U}$.
We are also discussing the other conditions appearing in the beginning of this article, e.g.\
	\begin{gather}
		m_k^2\leq m_{k-1}m_{k+1},\qquad \fa k\in\N,\label{StrongLogConvex}\\
		\ex C>0:\; M_{k+1}\leq C^{k+1}M_k\quad \fa k\in\N_0,\label{DClosed}
	\end{gather}

\begin{Rem}\label{WeightRemark}
	\begin{enumerate}
		\item If $\bM$ is a weight sequence satisfying 
		\eqref{AnalyticInclusion1} and \eqref{StrongLogConvex} then clearly
		$\mathbf{m}=(m_k)_k$ is also a weight sequence. In particular
		\begin{equation}\label{strongConcl}
			m_{j}m_k\leq m_{j+k}\qquad j,k\in\N_0.
		\end{equation}
	Furthermore, we have the following property:
	\begin{equation}\label{RootIncreasing}
		\text{The sequence }\sqrt[k]{m_k}=\frac{\Lambda_k}{k!^{1/k}} \text{ is increasing.}
	\end{equation}
Therefore the sequence $\Lambda_k$ is strictly increasing.
	 \item The estimate \eqref{DClosed} is equivalent to
	 \begin{equation*}
	 	\ex C>0:\; m_{k+1}\leq C^{k+1}m_k\qquad \fa k\in\N_0.
	\end{equation*}
	\end{enumerate}
\end{Rem}
For later use we note the following Lemma.
\begin{Lem}\label{StrongLemma}
	Let $\bM$ be a sequence with $M_0=1\leq M_1$ satisfying \eqref{AnalyticInclusion1}.
 Then the following statements hold:
	\begin{enumerate}
		\item If \eqref{StrongLogConvex} is satisfied then $\mu_k+1\leq \mu_{k+1}$ for all $k\in\N$.
	\item Assume that the sequence $\sqrt[k]{m_k}$ is increasing, i.e.~\eqref{RootIncreasing} holds.
	If we consider the sequence 
	$\Theta_k=k\sqrt[k]{m_k}$ then we have also that
	\begin{equation}
		\Theta_k+1\leq\Theta_{k+1}\qquad \fa k\in\N.\label{StrongCor}
	\end{equation}
	\end{enumerate}
\end{Lem}
\begin{proof}
	Recall from Remark \ref{WeightRemark} that $\mathbf{m}$ is a weight sequence
	and hence we can apply Lemma \ref{weakLemma}.
	By Lemma \ref{weakLemma}(3) we thus have that the sequence $\mu_k/k$ is increasing and therefore
	\begin{equation*}
		\mu_k+\frac{\mu_k}{k}\leq \mu_{k+1}\qquad \fa k\in\N.
	\end{equation*}
This gives instantly (1) since $\mu_k/k\geq 1$.

	For (2) we need only to observe that
	\begin{equation*}
		k\sqrt[k]{m_k}+1\leq k\sqrt[k+1]{m_{k+1}}+1\leq (k+1)\sqrt[k+1]{m_{k+1}}
	\end{equation*}
for all $k\in\N$ since $\sqrt[k]{m_k}\geq 1$.
\end{proof}
\begin{Rem}\label{StrongRemark}
	We recall that condition \eqref{AnalyticInclusion1} 
	implies that $\An(U)\subsetneq\DC{\bM}{U}$.
	However, it is nearly superfluous in view of \eqref{StrongLogConvex}:
	Iterating the estimate in Lemma \ref{StrongLemma}(1)
	 we obtain that for any strongly logarithmically convex weight 
	sequence $\bM$ 
	 we have that $k\leq \mu_k$ and therefore $k!\leq M_k$.
In particular $(k!)\preceq M_k$  which in turn implies
	that $\An(U)\subseteq\Rou{\bM}{U}$.
In fact, for that relation to hold it is enough to assume that $\sqrt[k]{m_k}$ is increasing,
since this gives also that $\liminf_{k\rightarrow\infty}\sqrt[k]{m_k}>0$ (recall that $m_1\geq 1$ by
assumption).

	In the Beurling situation, we know that \eqref{AnalyticInclusion1} holds if and only
	if $\An(\Real)\subseteq\Beu{\bM}{\Real}$. 
	Thus we only formally exclude the analytic case by our assumptions, which of course has been well investigated.
\end{Rem}

In order to deal later with the Beurling case we need the following statement, which will allow us to reduce the key argument in the Beurling case to the Roumieu case.
Its proof is based on the proof of \cite[Lemma 6]{MR550685},
cf.\ also \cite[Lemma 2.2]{zbMATH07538293}.
\begin{Prop}
\label{KomatsuTrick}
	Let $\bM$ be a weight sequence satisfying \eqref{AnalyticInclusion1} and \eqref{RootIncreasing}.
	If $\bL=(L_k)_k$ is a sequence of positive numbers
	such that $\bL\lhd\bM$
	then there is a sequence $\bN$ with $N_0=1\leq N_1$ satisfying \eqref{AnalyticInclusion1} and
	\eqref{RootIncreasing} such that 
	\begin{equation*}
		\bL\leq\bN\lhd\bM.
	\end{equation*}
\end{Prop}
\begin{proof}
	We set $m_k=M_k/k!$ and $\ell_k=L_k/k!$.
	Then the condition $\bL\lhd\bM$ implies that for all $h>0$ there is a constant $C_h$ such that
	\begin{equation}\label{lhd}
		\ell_k\leq C_h h^km_k\qquad \fa k\in\N_0.
	\end{equation}
For $h>0$ we take $C_h^\bullet$ to be the smallest number such that \eqref{lhd} holds.
We set
\begin{equation}\label{infimum}
	\tilde{\ell}_k=\inf_{h>0} C^\bullet_hh^km_k,\qquad k\in\N_0.
\end{equation}
and define a auxillary sequence by $\bar{\ell}_k= \tilde{\ell}_k/\tilde{\ell}_0$.
Since  the infimum in \eqref{infimum} is assumed for some $h>0$ we have that
\begin{equation*}
	\left(\frac{m_k}{\bar{\ell}_k}\right)^2=\frac{\tilde{\ell}_0}{C_h^\bullet h^{k-1}}\frac{\tilde{\ell}_0}{C_h^\bullet h^{k+1}}
	\leq \frac{m_{k-1}}{\bar{\ell}_{k-1}}\frac{m_{k+1}}{\bar{\ell}_{k+1}},\qquad k\in\N.
\end{equation*}
Hence the sequence $(m_k/\bar{\ell}_k)_k$ is logarithmic convex and by definition $\bar{\ell}_0= 1$
and therefore the sequence
\begin{equation*}
	c_k=\frac{\sqrt[k]{m_k}}{\sqrt[k]{\bar{\ell}_k}}
\end{equation*}
is increasing since Lemma \ref{weakLemma}(2) still holds in that situation.
Furthermore $c_k\rightarrow \infty$ because $\bL\lhd\bM$.

Now we define the sequence $\bN=(N_k)_k$ by $N_0=1$ and $N_k=k!n_k$ for $k\in\N$, where
\begin{equation*}
	\sqrt[k]{n_k}=\max\left\{\sqrt[2k]{m_k};\max_{j\leq k}\frac{\sqrt[j]{m_j}}{c_j}\right\}.
\end{equation*}
Clearly the sequence $\sqrt[k]{n_k}$ is increasing.
Moreover, $\sqrt[k]{n_k}\geq \sqrt[2k]{m_k}\rightarrow\infty$ for $k\rightarrow\infty$.
Thence the sequence $\bN$ satisfies the conditions \eqref{AnalyticInclusion1}
and \eqref{RootIncreasing}.

Finally we observe that $L_k\leq k!\bar{\ell}_k\leq k! n_k=N_k$
and
\begin{equation*}
	\sqrt[k]{\frac{n_k}{m_k}}=\max\left\{\frac{\sqrt[2k]{m_k}}{\sqrt[k]{m_k}};
	\max_{{j\leq k}}\frac{\sqrt[j]{m_j}}{c_j\sqrt[k]{m_k}}\right\}
	\longrightarrow 0.
\end{equation*}
Thence $\bL\leq\bN\lhd\bM$.
\end{proof}
\subsection{Weight matrices}
Now we are in the position to introduce the concept of weight matrices.
\begin{Def}
	A weight matrix $\fM$ is a family of weight sequences such that for all $\bM,\bN\in\fM$ we
	have either $\bM\leq\bN$ or $\bN\leq \bM$.
\end{Def}
If $\fM$ is a weight matrix then the Roumieu class of global ultradifferentiable functions associated
to $\fM$ is 
\begin{align*}
	\gRou{\fM}{U}&=\ind_{\substack{\bM\in\fM\\h>0}}\Ban{\bM}{U}\\
	\intertext{whereas the Beurling class is}
	\gBeu{\fM}{U}&=\proj_{\substack{\bM\in\fM\\h>0}}\Ban{\bM}{U}.\\
	\intertext{Then the local ultradifferentiable classes associated to $\fM$ are given by}
	\DC{\fM}{U}&=\proj_{V\Subset U}\gDC{\fM}{V}.
\end{align*}
If $\fM$ and $\fN$ are two weight matrices then we set
\begin{align*}
	\fM&\{\preceq\}\fN & &\Longleftrightarrow &  \fa \bM\in\fM \;\ex \bN\in\fM :\;\; \bM&\preceq\bN,\\
	\fM&(\preceq)\fN& &\Longleftrightarrow &\fa \bN\in\fN \; \ex\bM\in\fM :\;\;\bM&\preceq\bN,\\
	\fM&(\lhd\}\fN & &\Longleftrightarrow &\fa\bM\in\fM\;\fa \bN\in\fN:\;\;\bM&\lhd\bN.
\end{align*}
We write $[\preceq]=\{\preceq\},(\preceq)$. Clearly $\gDC{\fM}{U}\subseteq\gDC{\fN}{U}$,
$\DC{\fM}{\Omega}\subseteq\DC{\fN}{\Omega}$ if $\fM[\preceq]\fN$ and
$\gRou{\fM}{U}\subseteq\gBeu{\fN}{U}$, $\Rou{\fM}{\Omega}\subseteq\Beu{\fN}{\Omega}$ when
$\fM(\lhd\} \fN$.
We set also $\fM[\approx]\fN$ if $\fM[\preceq]\fN$ and $\fN[\preceq]\fM$.
Then $\gDC{\fM}{U}=\gDC{\fN}{U}$, $\DC{\fM}{\Omega}=\DC{\fN}{\Omega}$ for $\fM[\approx]\fN$.

\begin{Def}
	Let $\fM$ be a weight matrix. 
	\begin{enumerate}
		\item We say that $\fM$ is $R$-semiregular if the following conditions hold:
		\begin{gather}\label{AnalyticInclusion2}
			\fa\bM\in\fM:\quad \lim_{k\rightarrow\infty}\sqrt[k]{m_k}=\infty,\\
			\fa\bM\in\fM\ex\bN\in\fM\ex q>0:\;\, M_{k+1}\leq q^{k+1}N_k,\qquad
			k\in\N_0.
			\label{R-DerivClosed}
		\end{gather}
	\item $\fM$ is $B$-semiregular if \eqref{AnalyticInclusion2} 
	and 
	\begin{equation}\label{B-DerivClosed}
		\fa\bN\in\fM\ex\bM\in\fM\ex q>0:\;\, M_{k+1}\leq q^{k+1}N_k,\qquad k\in\N_0,
	\end{equation}
are satisfied.
	\end{enumerate}

We write 
$[$semiregular$]=R$-semiregular, $B$-semiregular.
\end{Def}
\begin{Rem}
	Let $d\in\N$ be fixed and $\fM$ be a weight matrix.
	\begin{enumerate}
		\item If $\fM$ is $R$-semiregular then the following holds:
		\begin{equation}\label{R-DerivClosed2}
			\fa \bM\in\fM\; \ex \bN\in\fM\; \ex q>0:\quad M_{k+d}\leq q^{k+1}N_k,\qquad \fa k\in\N_0.
		\end{equation}
	\item If $\fM$ is $B$-semiregular then we have that:
	\begin{equation}\label{B-DerivClosed2}
		\fa \bN\in\fM\; \ex \bM\in\fM\; \ex q>0:\quad M_{k+d}\leq q^{k+1}N_k,\qquad \fa k\in\N_0.
	\end{equation}
	\end{enumerate}
The statements follow by iterating \eqref{R-DerivClosed} and \eqref{B-DerivClosed}, respectively.
\end{Rem}

Let $U$ be an open set in ${\Real}^n$. We say that $U$ has Lipschitz boundary if
for all $x_0\in\partial U$ there are  some $r>0$, local coordinates $(x_1,\dotsc,x_n)$
and a Lipschitz function $h=h(x_1,\dotsc,x_{n-1})$ such that
\begin{equation*}
	U\cap B(x_0,r)=\Set*{(x_1,\dotsc,x_n)\given x_n> h(x_1,\dotsc,x_{n-1})}\cap B(x_0,r)
\end{equation*}
where $B(x_0,r)$ is the ball of radius $r$ in ${\Real}^n$ centered at $x_0$.
\begin{Rem}\label{Characterization}
	  For completeness we give an alternative characterization of $\DC{\fM}{\Omega}$ when
	  $\fM$ is a 
	$[$semiregular$]$ weight matrix.
	\begin{enumerate}
		\item Let $U\subseteq{\Real}^n$ be a bounded open set
	with Lipschitzian boundary. The Sobolev Theorem \cite{Adams1975} implies that
	a smooth function $f\in\E(U)$ is an element of $\gRou{\fM}{U}$
	(of $\gBeu{\fM}{U}$) if and only if there are constants $C,h>0$ and some $\bM\in\fM$
	(for every $h>0$ and $\bM\in\fM$ there is a constant $C>0$) such that
	\begin{equation}\label{Sobolev}
		\norm*{D^\alpha f}{L^2(U)}\leq Ch^\alp M_\alp,\qquad \fa \alpha\in\N_0^n.
	\end{equation}
We consider the more difficult Beurling case and leave the Roumieu case to the reader.
Hence let $f\in\E(U)$  and suppose that for all $\bM\in\fM$ and $h>0$ there is some $C>0$ such
that \eqref{Sobolev} is satisfied for all $\alpha\in\N_0^n$.
By the Sobolev Imbedding Theorem we have that for integers
$\sigma>n/2$ the following estimate holds
\begin{equation}\label{Obvious}
	\sup_{x\in U}\abs*{u(x)}\leq A\norm*{u}{H^\sigma(U)}
\end{equation}
for all $u\in\E(U)$ where the constant $A$ depends only on $U$, $n$ and $\sigma$.
We fix now $\sigma>n/2$.
If we set $u= D^\alpha f$ then we obtain
\begin{equation*}
	\begin{split}
	\norm*{D^\alpha f}{H^\sigma(U)}^2&=\sum_{\bet\leq \sigma}\norm*{D^{\alpha+\beta}f}{L^2(U)}^2\\
	&\leq C^2 \sum_{\bet\leq \sigma}h^{2(\alp+\bet)}M_{\alp+\bet}^2\\
	&\leq C^2 h^{2\alp}M_{\alp+\sigma}^2\sum_{\bet\leq \sigma}h^\bet.
\end{split}
\end{equation*}
Thence, using \eqref{B-DerivClosed2} we conclude that for all $\bM\in\fM$ and $h>0$ there is
a constant $C>0$ such that 
\begin{equation*}
	\norm*{D^\alpha f}{H^\sigma(U)}\leq Ch^\alp M_\alp,\qquad \alp\in\N_0^n.
\end{equation*}
Thus applying also \eqref{Obvious} gives that $f\in\gBeu{\fM}{U}$.
The other direction follows from the trivial estimate
\begin{equation*}
	\norm*{g}{L^2(U)}\leq \abs{U}\sup_{x\in U}\abs{g(x)},\qquad \fa g\in\CC(U).
\end{equation*}
\item
On the other hand, it is easy to see that, if $\Omega\subseteq{\Real}^n$ is an arbitrary open set and
$f\in\E(\Omega)$ then   $f\in\DC{\fM}{\Omega}$
 if and only if for all $x\in\Omega$ there is
a neighborhood $U$ of $x$ such that $f\vert_U\in\gDC{\fM}{U}$.

\item Combining these two statements we have the following characterizations:
Let $f\in\E(\Omega)$.
\begin{itemize}
	\item If $\fM$ is an $R$-semiregular weight matrix then $f\in\Rou{\fM}{\Omega}$
	if and only if for all $x\in\Omega$ there is a neighborhood $U$ of $x$ such that
	\begin{equation*}
	\ex\bM\in\fM\,	\ex h>0\,\ex C>0:\quad \norm*{D^\alpha f}{L^2(U)}\leq Ch^\alp M_\alp\qquad
	\fa\alpha\in\N_0^n.
	\end{equation*}
\item If $\fM$ is a $B$-semiregular weight matrix then $f\in\Beu{\fM}{\Omega}$ if and only if 
for all $x\in\Omega$ there is a neighborhood $U$ of $x$ such that
\begin{equation*}
	\fa \bM\in\fM\,\fa h>0\,\ex C>0:\quad \norm*{D^\alpha f}{L^2(U)}\leq Ch^\alp M_\alp\qquad
	\fa\alpha\in\N_0^n.
\end{equation*}
\end{itemize}
\end{enumerate}
\end{Rem}
\begin{Def}
	Let $\fM$ be a weight matrix.
	\begin{enumerate}
		\item We say that $\fM$ is weakly $R$-regular if \eqref{AnalyticInclusion2},
		\eqref{R-DerivClosed} and
		\begin{equation}
				\label{RootIncreasing2}
			\fa\bM\in\fM:\quad \text{The sequence }\sqrt[k]{m_k}\text{ is increasing.}
			\end{equation}
	 \item $\fM$ is $B$-regular if $\fM$ satisfies \eqref{AnalyticInclusion2}, 
	 \eqref{B-DerivClosed} and  
	 \eqref{RootIncreasing2}.
	\end{enumerate}
\end{Def}
We write $[$weakly regular$]=$weakly $R$-regular, weakly $B$-regular.
Clearly, every $[$weakly regular$]$ weight matrix $\fM$ is $[$semiregular$]$.


  \begin{Rem}\label{OmegaRemark}
If $\omega$ is a weight function then the weight matrix $\fW$ associated to $\omega$ consists of
the weight sequences $\bW^\lambda$, $\lambda>0$, which are given by
\begin{equation*}
W_k^\lambda=e^{\frac{1}{\lambda}\varphi_\omega^\ast(\lambda k)}.
\end{equation*}
It is easy to see that $\DC{\omega}{U}=\DC{\fW}{U}$ as topological vector spaces.
If $\omega(t)=o(t)$ then $\fW$ is $R$- and $B$-semiregular.
Furthermore, \eqref{R-DerivClosed} and \eqref{B-DerivClosed} hold both for $\fW$,
cf. \cite{MR3285413}.

When $\omega$ is concave then there exists a weight matrix $\fT$ such that
$\fT\{\approx\}\fW$ and $\fT(\approx)\fW$ and
moreover $\fT$ satisfies \eqref{RootIncreasing2}; 
in fact, every weight sequence
$\mathbf{T}\in\fT$ satisfies \eqref{StrongLogConvex}, see \cite[Proposition 3]{RainerSchindl2020}.
Moreover, $\fT$ satisfies \eqref{AnalyticInclusion2}, since condition \eqref{AnalyticInclusion2}
is clearly invariant under $[\approx]$.
Thus $\DC{\omega}{\Omega}=\DC{\fT}{\Omega}$ and $\fT$ is weakly $R$- and weakly $B$-regular since
the conditions
\eqref{R-DerivClosed} and \eqref{B-DerivClosed} are also invariant 
under the equivalence relations $\{\approx\}$ and $(\approx)$, respectively.
\end{Rem}
\subsection{Ultradifferentiable vectors associated to weight matrices}
\begin{Def}\label{VectorDef}
	Let $\sP=\{P_1,\dotsc,P_\ell\}$ be a system of smooth differential operators in $U$,
	$d_j$ be the order of $P_j$ for $j=1,\dotsc,\ell$,
	$u\in\E(U)$ and $\fM$ be a weight matrix.
	\begin{enumerate}
			\item Assume that $a_{j\lambda}\in\gRou{\fM}{U}$ for all $\lamb\leq d_j$
			and every $j\in\{1,\dotsc,\ell\}$. Then we say that $u$ is a global vector of 
			class $\{\fM\}$ if there are $\bM\in\fM$ and $C,h>0$ such that
			\begin{equation}\label{GlobalVectorEstimate}
				\norm*{P^\tau u}{L^2(U)}\leq Ch^{\lvert\tau\rvert} M_{d_\tau}
			\end{equation}
		for all $\tau\in\{1,\dotsc,\ell\}^k$, $d_\tau=\sum_{j=1}^k d_{\tau_j}$
		and all $k\in\N_0$.
			\item When $a_{j\lambda}\in\gBeu{\fM}{U}$ for all $\lamb\leq d_j$ and every $j\in\{1,\dotsc,\ell\}$ then $u$ is a global vector of class $(\fM)$ if for all
			$\bM\in\fM$ and all $h>0$ there is a constant $C>0$ such that
			\eqref{GlobalVectorEstimate} holds for all $\tau\in\{1,\dotsc,\ell\}^k$ and every $k\in\N_0$.
		\item Assume that $a_{j\lambda}\in\Rou{\fM}{U}$ for all $\lamb\leq d_j$
		and every $j\in\{1,\dotsc,\ell\}$. Then we say that $u$ is a (local) vector of 
		class $\{\fM\}$ if for all $V\Subset U$ there are $\bM\in\fM$ and $C,h>0$ such that
		\begin{equation}\label{LocalEstimate}
			\norm{P^\alpha u}{L^2(V)}\leq Ch^\alp M_\alp
		\end{equation}
		for $\alpha\in\{1,\dotsc,\ell\}^k$ and $k\in\N_0$.
		\item When $a_{j\lambda}\in\Beu{\fM}{U}$ for all $\lamb\leq d_j$ and every $j\in\{1,\dotsc,\ell\}$ then $u$ is a local vector of class $(\fM)$ if for all
		$V\Subset U$, all
		$\bM\in\fM$ and all $h>0$ there is a constant $C>0$ such that
		\eqref{LocalEstimate} holds for all $\alpha\in\{1,\dotsc,\ell\}^k$ and every $k\in\N_0$.
	\end{enumerate}
We denote the space of global vectors of class $[\fM]$ in $U$ by $\vgDC[\sP]{\fM}{U}$ and
$\vDC[\sP]{\fM}{U}$ is the space of all local vectors of class $[\fM]$.
\end{Def}
\begin{Prop}
	Let $\fM$ be a weight matrix and $\sP=\{P_1,\dotsc,P_\ell\}$ be a system
	of smooth differential operators in an open set $U\subseteq{\Real}^n$. Then the following holds:
\begin{enumerate}
	\item If the coefficients of the operators $P_j$ are all in $\DC{\fM}{U}$
	 then $\DC{\fM}{U}\subseteq\vDC[\sP]{\fM}{U}$.
	 \item If $U$ is bounded and the coefficients of the operators $P_j$ are all in $\gDC{\fM}{U}$
	 then $\gDC{\fM}{U}\subseteq\vgDC[\sP]{\fM}{U}$.
\end{enumerate}
\end{Prop}
\begin{proof}
Suppose that $Q(x,D)=\sum_{\alp\leq d}b_{\alpha}D^\alpha$ is a partial differential operator
with coefficients $b_\alpha\in\gRou{\fM}{U}$ and $f\in\gRou{\fM}{U}$.
Then we can assume that there exist some $h>0$ and $\bM\in\fM$
  such that $b_\alpha,f\in \Ban{\bM}{U}$. It follows that
\begin{equation*}
	\begin{split}
	\norm*{D^\beta Qf}{L^2(U)}&\leq C\sum_{\alp\leq d}\sum_{\gamma\leq\beta}\binom{\beta}{\gamma}
	h^{\bet-\gam}M_{\bet-\gam}h^{\gam+\alp}M_{\gam+\alp}\\
	&\leq C^\prime (h^\prime)^{\bet+d} M_{\bet+d}
\end{split}
\end{equation*}
for some constants $C^\prime>0$ and $h^\prime>0$.
From this estimate we can conclude that $\gRou{\fM}{U}\subseteq\vgRou[\sP]{\fM}{U}$.

In the Beurling case we have that $b_\alpha,f\in\gBeu{\fM}{U}=\bigcap_{\bM\in\fM}\gBeu{\bM}{U}$.
We define a sequence $\bL$ by
\begin{equation*}
	L_k=\max\left\{\sup_{x\in U} \abs*{D^\beta f(x)};\sup_{x\in U} \abs*{D^\beta b_\alpha(x)}:\; \alp\leq d,\,\bet\leq k\right\}.
\end{equation*}
If $\bM\in\fM$ is arbitrary then $\bL\lhd\bM$.
According to the proof of \cite[Lemma 2.2]{zbMATH07538293} there is a weight sequence 
$\bN$ such that $\bL\leq \bN\lhd\bM$. 
Thus $b_\alpha,f\in\Ban{\bN}{U}$ for some $h$. The estimate above gives
\begin{equation*}
\norm*{D^\beta Qf}{L^2(U)}\leq C^\prime (h^\prime)^{\bet+d}N_{\bet+d}
\end{equation*}
for some $C^\prime,h^\prime>0$. Hence for all $h>0$ there is a constant $C>0$
such that
\begin{equation*}
	\norm*{D^\beta Qf}{L^2(U)}\leq Ch^{\bet+d}M_{\bet+d}.
\end{equation*}
Since $\bM\in\fM$ was chosen arbitrarily we conclude that $\gBeu{\fM}{U}\subseteq\vgBeu[\sP]{\fM}{U}$.

The local case follows analogously.
\end{proof}
We can now state our main theorem:
\begin{Thm}\label{localThmIterates}
Let $\fM$ be a $[$weakly regular$]$ weight matrix and $\sP=\{P_1,\dotsc,P_\ell\}$ be an elliptic system of differential operators of class $[\fM]$ in $\Omega$. 
Then
\begin{equation*}
 \vDC[\sP]{\fM}{\Omega}=\DC{\fM}{\Omega}.
\end{equation*}
\end{Thm}
\begin{Rem}\label{OrderRem}
	Note that we can assume that all the operators $P_j\in\sP$ in Theorem \ref{localThmIterates} have the same order.
	Indeed, if $d_j$ denotes the order of $P_j$ for all $j\in\{1,\dotsc,\ell\}$ then we set 
	$d^\prime_j=\prod_{i\neq j}d_i$ and $Q_j=P_j^{d^\prime_j}$. Then all $Q_j$ are of order
	$d=\prod_j d_j$ and it is easy to see that the system $\mathsf{Q}=\{Q_1,\dotsc,Q_\ell\}$ is
	elliptic if and only if $\sP$ is elliptic. Furthermore
	$\vDC[\sP]{\fM}{\Omega}\subseteq\vDC[\mathsf{Q}]{\fM}{\Omega}$ for all weight matrices $\fM$.
	Hence, Theorem \ref{localThmIterates} will be proven if we show
	$\vDC[\mathsf{Q}]{\fM}{\Omega}\subseteq\DC{\fM}{\Omega}$.
\end{Rem}

\section{The fundamental estimate}\label{Sec:Estimates}
By Remark \ref{OrderRem} we will assume in this and the next section that all operators
$P_j\in\sP$ are of the same order $d\in\N$. We will also
 denote the open ball of radius
$R$ centered at 
a point $x\in{\Real}^n$ by $B(x,R)$.

We follow closely the structure of the proof given in \cite{MR548225} (see also \cite{MR1037999}). 

Beginning with a well-known a-priori estimate for elliptic systems of smooth operators 
of equal order $d$ (see e.g.\ \cite{Agmon1965}), i.e.\ 
for all $U\Subset \Omega$   there is a constant $C>0$ such that
\begin{equation*}
	\norm*{u}{H^k(U)}\leq C\left(\sum_{j=1}^\ell \norm*{P_ju}{H^{k-d}(U)}
	+\norm*{u}{L^2(U)}\right),\qquad\quad u\in\D(U),\; k=0,\dotsc,d,
\end{equation*}
we can deduce two other estimates following the arguments in \cite{MR548225}:
\begin{Prop}[cf.~{\cite[Proposition I-2]{MR548225}}]\label{FirstProp}
Let $\sP$ be an elliptic system of differential operators with equal order $d$ and
smooth coefficients in $\Omega$ and let $W^\prime\Subset W\Subset \Omega$ 
be open sets. Then there exists a constant $C>0$ such that
\begin{equation}\label{localIneq1}
\norm*{u}{H^d(W^\prime)}\leq C\left(\sum_{j=1}^\ell\norm*{P_ju}{L^2(W)}
+\norm*{u}{L^2(W)}\right)
\end{equation}
for all $u\in\E(\Omega)$.
\end{Prop}
\begin{Prop}[cf.~{\cite[Proposition I-3]{MR548225}}]
Let $\sP$ be as in Proposition \ref{FirstProp}, $x\in\Omega$ and $R,R_1>0$ with $R<R_1$  such that 
$B(x;R_1)\Subset\Omega$. For $\rho< R$ write $W[\rho]=B(x;R-\rho)$.
Then there exists a constant $C>0$ such that for all $u\in\E(\Omega)$,
for all $\alpha\in\N_0^n$ with $\alp\leq d$ and for all 
$\rho,\rho^\prime>0$ with $\rho+\rho^\prime<R$ we have
\begin{equation}\label{localIneq2}
\rho^d\norm*{D^\alpha u}{L^2(W[\rho+\rho^\prime])}\leq C\left(
\rho^d\sum_{j=1}^\ell\norm*{P_ju}{L^2(W[\rho^\prime])}
+\sum_{\bet\leq d-1}\rho^\bet\norm*{D^\beta u}{L^2(W[\rho^\prime])}
\right).
\end{equation}
\end{Prop}
We are now in the position to formulate and prove the main estimate which will be used in
the proof of Theorem \ref{localThmIterates}.
Note that if $\rho_1<\rho_2<R$ then $W[\rho_2]\Subset W[\rho_1]$. 
We may also set $W[\rho]=\emptyset$ if $\rho>R$.
Moreover, we recall also that for a weight sequence $\bM$ we have defined
the auxiliary sequence $(\Theta_k)_k$ by $\Theta_k=k\sqrt[k]{m_k}$. 
\begin{Prop}\label{MainProp}
Let $\bM$ be a sequence with $M_0=1\leq M_1$ satisfying \eqref{AnalyticInclusion1} and 
\eqref{StrongLogConvex}, $x\in\Omega$, 
$R<R_1\leq 1$ such that $W=B(x,R_1)\Subset\Omega$
and assume $\sP=\{P_1,\dotsc,P_\ell\}$  is an elliptic system of smooth differential operators 
$P_j=\sum_{\lamb\leq d} a_{j\lambda}D^\lambda$ on $\Omega$ such that
$a_{j\lambda}\vert_W\in\gRou{\bM}{W}$.
Then there exists a constant $A>0$ such that for all $0<\rho<R$, all 
$u\in\E(\Omega)$ and all $k\in\N$ we have that
\begin{equation}\label{localIneq3}
	\rho^{\alp}\norm*{D^\alpha u}{L^2(W[\Theta_{\alp}\rho])}
	\leq A^{\alp+1} S_k(u)
\end{equation}
where $\alp\leq dk$ and
\begin{equation*}
	S_k(u)=\sum_{\sigma=1}^{k}\rho^{(\sigma-1)d}
	\sum_{\tau\in\{1,\dotsc,\ell\}^{\sigma}}
	\norm{P^\tau u}{L^{2}(W)}+\norm{u}{L^2(W)}.
\end{equation*}
\end{Prop}
\begin{proof}
We begin by observing that $S_{k}(u)\leq S_{k+1}(u)$ and
\begin{equation*}
	\sum_{j=1}^\ell\rho^d S_k \left(P_ju \right)\leq S_{k+1}(u)
\end{equation*}
for all $k\in\N$ since $\rho\leq 1$ by assumption.

Furthermore there exists a constant $H>0$ such that
\begin{equation*}
	\norm*{D^\gamma a_{j,\lambda}}{L^2(W)}\leq H^{\gam+1} M_\gam
\end{equation*}
for all $\gamma\in\N_0^n$ and $\lambda\in\N_0^n$ with $\lamb\leq d$ and every $j\in\{1,\dotsc,\ell\}$.

We are going to prove \eqref{localIneq3} by induction in $k$.
To begin with, \eqref{localIneq1} implies that
\begin{equation*}
	\norm*{D^\alpha u}{L^2(W[0])}\leq A\left(\sum_{j=1}^\ell\norm*{P_ju}{L^2(W)}
	+\norm{u}{L^2(W)}	\right)
\end{equation*}
for $\alp\leq d$ which gives \eqref{localIneq3} for $k=1$ since $\rho\leq 1$ 
and without loss of generality we can assume that $A\geq 1$.

Now let $\alpha\in\N_0^n$ be such that $dk<\alp\leq d(k+1)$ and assume that
\eqref{localIneq3} has been shown for all $\beta\in\N_0^n$ with $\bet\leq\alp -1$.
We put $\alpha=\alpha_0+\alpha^\prime$ with $\lvert\alpha_0\rvert=d$.
If we replace in \eqref{localIneq2} $\rho^\prime$ by $(\Theta_{\alp} -1)\rho$,
$\alpha$ by $\alpha_0$ and $u$ by $D^{\alpha^\prime}u$, then we obtain
\begin{equation*}
	\rho^{\alp}\norm*{D^\alpha u}{L^2(W[\Theta_{\alp}\rho])}	
\begin{multlined}[t]
\leq C\left\{\rho^{\alp}\sum_{j=1}^\ell
\norm*{P_j\bigl(D^{\alpha^\prime}u\bigr)}{L^2(W[(\Theta_{\alp}-1)\rho])}\right.\\
\qquad\qquad\left.+\sum_{\bet\leq d-1}\rho^{\alp -d+\bet}
\norm*{D^{\beta+\alpha^\prime}u}{L^2(W[(\Theta_\alp-1)\rho])}
\right\}.
\end{multlined}
\end{equation*}
We observe that
\begin{equation*}
	D^{\alpha^\prime}(P_j u)-P_j\Bigl(D^{\alpha^\prime}u\Bigr)
	=\sum_{\lamb\leq d}
	\sum_{\substack{\gamma\leq\alpha^\prime\\\gamma\neq\alpha^\prime}}
	\binom{\alpha^\prime}{\gamma}D^{\alpha^\prime-\gamma}a_{j\lambda}
	D^{\lambda+\gamma}u
\end{equation*}
can be estimated by
\begin{multline}\label{RegularEstimate2}
	\norm*{D^{\alpha^\prime}(P_j u)
		-P_j\left(D^{\alpha^\prime}u\right)}{L^2(W[(\Theta_{\alp}-1)\rho])}
	\\
	\leq
	\sum_{\lamb\leq d}
	\sum_{\substack{\gamma\leq\alpha^\prime\\\gamma\neq\alpha^\prime}}
	\binom{\lvert\alpha^\prime\rvert}{\gam}
	H^{\lvert\alpha^\prime-\gamma\rvert+1}M_{\lvert\alpha^\prime-\gamma\rvert}
\bigl(\Theta_{\lvert\alpha^\prime\rvert}\rho\bigr)^{-\lvert\alpha^\prime-\gamma\rvert}
\norm*{D^{\beta+\gamma}u}{L^2(W[(\Theta_{\alp}-1)\rho])}
\end{multline}
since if $\Theta_{\lvert\alpha^\prime\rvert}\rho>R$ 
then $W[(\Theta_{\alp}-1)\rho]=\emptyset$, because $\Lambda_{\alp}-\Lambda_{\lvert\alpha^\prime\rvert}\geq d$
by \eqref{StrongCor}. Now note that $\Lambda_k\leq\Theta_k$.
Furthermore, using also the fact that $\binom{k}{j}\leq k^j/(k!)$ for $0\leq j\leq k$
we conclude that
\begin{equation}\label{RegularEstimate}
	\begin{split}
		\binom{\lvert\alpha^\prime\rvert}{\lvert\alpha^\prime-\gamma\rvert}
		M_{\lvert\alpha^\prime-\gamma\rvert}\Theta_{\lvert\alpha^\prime\rvert}^{
		-\lvert\alpha^\prime-\gamma\rvert}
	&\leq \binom{\lvert\alpha^\prime\rvert}{\lvert\alpha^\prime-\gamma\rvert} \lvert\alpha^\prime-\gamma\rvert! m_{\lvert\alpha^\prime-\gamma\rvert}
\lvert\alpha^\prime\rvert^{-\lvert\alpha^\prime-\gamma\rvert}
\left(m_{\lvert\alpha^\prime\rvert^{1/\lvert\alpha^\prime\rvert}}\right)^{-\lvert\alpha^\prime-\gamma\rvert}\\
	&\leq \left(\frac{m_{\lvert\alpha^\prime-\gamma\rvert}^{1/\lvert\alpha^\prime-\gamma\rvert}}{
	m_{\lvert\alpha^\prime\rvert}^{1/\lvert\alpha^\prime\rvert}}\right)
\prod_{j=1}^{\lvert\alpha^\prime-\gamma\rvert}\frac{\lvert\gamma\rvert+j}{\lvert\alpha^\prime\rvert}
\leq 1
\end{split}
\end{equation}
since $\sqrt[k]{m_k}$ is increasing.
Therefore
\begin{equation*}
	\rho^{\alp}\norm{D^\alpha u}{L^2(W_{\Theta_{\alp}\rho})}
	\leq
	\begin{multlined}[t][10cm]
		 C\left\{\rho^{\alp}\sum_{j=1}^\ell
		\norm*{D^{\alpha^\prime}\bigl(P_ju\bigr)}{L^2(W[\Theta_{\lvert\alpha^\prime\rvert}\rho])}
			\right.\\
 +\sum_{\lamb\leq d}\sum_{\substack{\gamma\leq\alpha^\prime\\\gamma\neq\alpha^\prime}}
 H^{\lvert\alpha^\prime\rvert-\gam +1}\rho^{\alp-\lvert\alpha^\prime\rvert+\gam}
 \norm*{D^{\lambda+\gamma}u}{L^2(W[\Theta_{(d+\gam)}\rho])}\\ 
 \left.+\sum_{\bet\leq d-1}\rho^{\alp -d+\bet}
 \norm*{D^{\beta+\alpha^\prime}u}{L^2(W[\Theta_{\bet+\lvert\alpha^\prime\rvert}\rho])}
 \right\},
	\end{multlined} 
\end{equation*}
since $W[(\Theta_k-1)\rho]\subseteq W[\Theta_\ell\rho]$ by \eqref{StrongCor} for all $\ell<k$.
Now, the induction hypothesis implies the following estimates:
\begin{align*}
	\rho^{\alp}\sum_{j=1}^\ell
	\norm*{D^{\alpha^\prime}
		\bigl(P_ju\bigr)}{L^2(W[\Theta_{\lvert\alpha^\prime\rvert}\rho])}
	&\leq \rho^d A^{\lvert\alpha^\prime\rvert+1}\sum_{j=1}^\ell
	S_k(P_ju)\leq A^{\lvert\alpha^\prime\rvert+1} S_{k+1}(u),\\
	\sum_{\substack{\lamb\leq d\\\gamma\leq\alpha^\prime\\\gamma\neq\alpha^\prime}}
	H^{\lvert\alpha^\prime\rvert-\gam +1}\rho^{d+\gam}
	\norm*{D^{\lambda+\gamma}u}{L^2(W[\Theta_{(d+\gam)}\rho])}
	&\leq\sum_{\substack{\lamb\leq d\\\gamma\leq\alpha^\prime\\\gamma\neq\alpha^\prime}}
		H^{\lvert\alpha^\prime\rvert-\gam +1} A^{d+\gam+1}S_{k+1}(u),\\
	\sum_{\bet\leq d-1}\rho^{\alp -d+\bet}
	\norm*{D^{\beta+\alpha^\prime}u}{L^2(W[\Theta_{\bet+\lvert\alpha^\prime\rvert}\rho])}
	&\leq \sum_{\bet\leq d-1}A^{\lvert\alpha^\prime\rvert+\bet+1}S_{k+1}(u).
\end{align*}
Hence we have obtained that
\begin{equation}
	\begin{split}\label{Estimate2}
	\rho^{\alp}\norm*{D^\alpha u}{L^2(W[\Theta_{\alp}\rho])}
	&\leq A^{\alp +1}S_{k+1}(u)\times\\
		&\qquad\times\left\{CA^{-d}+C\sum_{\lamb\leq d}
	\sum_{\substack{\gamma\leq\alpha^\prime\\ \gamma\neq\alpha^\prime}}
	H^{\lvert\alpha^\prime\rvert-\gam+1}
	A^{\gam-\lvert\alpha^\prime\rvert}
	+C\sum_{\bet\leq d-1}A^{\bet-d}\right\}.
\end{split}
\end{equation}
Since 
\begin{equation*}
	C\sum_{\lamb\leq d}
	\sum_{\substack{\gamma\leq\alpha^\prime\\ \gamma\neq\alpha^\prime}}
	H^{\lvert\alpha^\prime\rvert-\gam+1}
	A^{\gam-\lvert\alpha^\prime\rvert}
	\leq Cd^nH^2A^{-1}\sum_{\beta\in\N^n_0}\bigl(HA^{-1}\bigr)^\bet,
\end{equation*}
we are able to choose $A$ large enough and independent of $\alpha$ and $\rho$
 so that the bracket on the right-hand side
of \eqref{Estimate2} is $\leq 1$.
\end{proof}

\section{Proof of Theorem \ref{localThmIterates}}\label{Sec:Proof}

\subsection{The Roumieu case}\label{RoumieuProof}
Let $\fM$ be an $R$-regular weight matrix and $u\in\vRou[\sP]{\fM}{\Omega}$. 
We have to prove that for all $x\in\Omega$
there is a neighborhood $U\Subset\Omega$ of $x$ such that $u\vert_U\in\gRou{\fM}{U}$.

Therefore we fix $x_0\in\Omega$ and choose $R_1\leq 1$ such
that $W=B(x_0;R_1)\Subset \Omega$. 
Then by assumptation there are  a weight sequence $\bM\in\fM$, 
satisfying \eqref{AnalyticInclusion1} and \eqref{StrongLogConvex},
 and constants $C,h>0$ such that
 \begin{equation*}
 \norm*{P^\tau u}{L^2(W)}\leq Ch^k M_{dk}
 \end{equation*}
 for all $\tau\in\{1,\dotsc,\ell\}^k$ and all $k\in\N$.
 
We conclude that
 \begin{equation*}
 S_k(u)\leq C\sum_{\sigma=1}^k\rho^{(\sigma-1)d}\ell^\sigma h^\sigma M_{d\sigma}+C
 \end{equation*}
 for all $0<\rho<R$ for some $0<R<R_1$.
 Hence by \eqref{localIneq3} we have that
 \begin{equation*}
 	\norm*{D^\alpha u}{L^2(W[\mu_\alp \rho])}
 	\leq CA^{\alp+1}\left(\sum_{\sigma=1}^k\rho^{d(\sigma-1)-\alp}\ell^\sigma h^\sigma M_{d\sigma}+1\right)
 \end{equation*}
for every $0<\rho< R$, all $\alpha\in\N_0^n$ with $d(k-1)<\alp\leq dk$ for $k\in\N$ where $C,A,h>0$ are constants independent of $\rho$, $k$ and $\alpha$. 

Now choose some $R^\prime$ with $0<R^\prime<R<R_1<1$. In particular $R-R^\prime<1$ and 
we set 
\begin{equation*}
	\rho=\frac{R-R^\prime}{e\Lambda_{dk}}.
\end{equation*}
Thus, 
we have for $d(k-1)\leq\alp\leq dk$ that $dk\leq \alp+d$ and therefore
\begin{equation*}
	\begin{split}
	M_{d\sigma}\rho^{d(\sigma-1)-\alp}
	&=\left(\Lambda_{d\sigma}\right)^{d\sigma} \left(\frac{R-R^\prime}{e}\right)^{d(\sigma-1)-\alp}
\left(\Lambda_{dk}\right)^{\alp-d(\sigma-1)}\\
&\leq\left(\frac{R-R^\prime}{e}\right)^{d(\sigma-1)-\alp}\Lambda_{dk}^{\alp+d}\\
&\leq R_2^{\alp -d(\sigma-1)}M_{\alp+d}
\end{split}
\end{equation*}
where $R_2=e(R-R^\prime)^{-1}>1$.

Now, note that the sequence $\Theta_k$ is strictly increasing and
 the Stirling formula implies that
\begin{equation*}
	\frac{1}{e}\leq\frac{\Theta_k}{e\Lambda_k}= \frac{k}{e(k!)^{1/k}}
	\leq \frac{1}{(2\pi k)^{1/2k}}\leq 1
\end{equation*}
for all $k\in\N$.
Thence we have the following estimate
\begin{equation*}
	\begin{split}
	R-\Theta_\alp\rho&\geq
	R-\Theta_{dk}\rho\\
	&= R\left(1-\frac{\Theta_{dk}}{e\Lambda_{dk}}\right)+\frac{\Theta_{dk}}{e\Lambda_{dk}}R^\prime\\
	&\geq e^{-1}R^\prime
\end{split}
\end{equation*}
and therefore $U=B(x_0,e^{-1}R^\prime)\subseteq W[\Lambda_\alp\rho]$.
Thus we can, if we enlarge $h$ when necessary, estimate that
\begin{equation*}
	\begin{split}
		\norm*{D^\alpha u}{L^2(U)}
		&\leq CA^{\alp+1}\left(\sum_{\sigma=1}^kR_2^{\alp-d(\sigma-1)}\ell^\sigma h^\sigma M_{\alp+d}+1\right)\\
		&\leq CA^{\alp+1}(\ell h)^{(\alp+d)/d} R_2^\alp M_{\alp+d} \sum_{\sigma=0}^k1\\
	\end{split}
\end{equation*}
for every $k\in\N$ and all $\alpha\in\N_0^n$ with $d(k-1)<\alp\leq dk$.

Since $d$ does not depend on $\alpha$ or $k$ we have by \eqref{R-DerivClosed2} that
there is a weight sequence $\bM^\prime$ and constants $C_1,h_1$ such that
%
\begin{equation*}
\norm*{D^\alpha u}{L^2(U)}\leq C_1 h_1^{\alp} M^\prime_{\alp}
\end{equation*}
for all $\alpha\in\N_0^n$.
Thus $u\in\Rou{\fM}{\Omega}$ by Remark \ref{Characterization}(3).

\subsection{The Beurling case}\label{BeurlingProof}
Now we assume that $\fM$ is weakly $B$-regular and $u\in\vBeu[\sP]{\fM}{U}$.
Note that
\begin{equation*}
	\Beu{\fM}{\Omega}=\bigcap_{\bM\in\fM}\Beu{\bM}{\Omega},\qquad
	\vBeu[\sP]{\fM}{\Omega}=\bigcap_{\bM\in\fM}\vBeu[\sP]{\bM}{\Omega}.
\end{equation*}
Thus we consider first the case where $u\in\vBeu[\sP]{\bM}{\Omega}$
and $a_{j\lambda}\in\Beu{\bM}{\Omega}$, $1\leq j\leq\ell$, $\lamb\leq d$,
with $\bM$ being a weight sequence for which \eqref{AnalyticInclusion1} and
\eqref{RootIncreasing} hold.
We fix $x\in\Omega$ and let $0<R<R_1\leq 1$ be such that  $W=B(x;R_1)\Subset\Omega$.
We define a sequence $\bL$ by setting
\begin{equation*}
	\begin{split}
	L_k&=\max\Bigl\{k!;\;\sup_{x\in W}
	\abs*{D^\alpha a_{j,\lambda}}:\,\alp\leq k,\lamb\leq d, j\in\{1,\dotsc,\ell\}
	;\\
	&\qquad\qquad\qquad\qquad\qquad\norm*{P^\tau u}{L^2(W)}:\, \tau\in\{1,\dotsc,\ell\}^{\nu},
	\nu\leq \tfrac{k}{d} \Bigr\}.
\end{split}
\end{equation*}
According to Lemma \ref{KomatsuTrick} there is a  sequence $\bN$ with $N_0=1\leq N_1$
 satisfying \eqref{AnalyticInclusion1} and \eqref{RootIncreasing} such
that $\bL\preceq\bN\lhd\bM$.
Hence $u\in\vgRou[\sP]{\bN}{W}$ and $a_{j\lambda}\in\gRou{\bN}{W}$.
It follows that we can apply Proposition \ref{MainProp} and obtain that there is a 
constant $A$ such that for all $k\in\N$ and every $\alpha$ with $d(k-1)<\alp\leq dk$ we have
\begin{equation*}
	\rho^{\alp}\norm{D^\alpha u}{L^2(W_{\nu_\alp\rho})}\leq A^{\alp +1}S_k(u)
\end{equation*}
where $S_k(u)$ is as in Proposition \ref{MainProp} with $\bM$ replaced by $\bN$ and $0<\rho<R$ is
chosen arbitrarily but fixed.
 Since \eqref{RootIncreasing} still implies that $\bN$ and $\sqrt[k]{N_k}$ are increasing,
 the arguments in the previous subsection yield that
there are a neighborhood $U$ of $x_0$ and constants $C_1,h_1>0$ such that
\begin{equation*}
	\norm{D^\alpha u}{L^2(U)}\leq C_1h_1^\alp N_{\alp+d}
\end{equation*}
for all $k\in\N$ and all $d(k-1)<\alp\leq dk$. Since $\bN\lhd\bM$ we conclude that
for all $h>0$ there is a constant $C>0$ such that
\begin{equation*}
	\norm*{D^\alpha u}{L^2(U)}\leq Ch^{\alp} M_{\alp+d}.
\end{equation*}
But $\bM\in\fM$ has been chosen arbitrarily and therefore we obtain the above estimate for
all $\bM\in\fM$ if $u\in\vBeu[\sP]{\fM}{\Omega}$.
Now we can employ \eqref{B-DerivClosed2} to conclude that for all weight sequences $\bM^\prime$ and $h_1>0$
there is a constant $C_1>0$ such that
\begin{equation*}
	\norm*{D^\alpha u}{L^2(U)}\leq C_1h_1^{\alp} M^\prime_{\alp+d}
\end{equation*}
Applying Remark \ref{Characterization}(3) we observe that $u\in\Beu{\fM}{\Omega}$. 

\section{Remarks}\label{Remarks}
\subsection{Elliptic regularity in ultradifferentiable classes}
Let $\fM$ be a weight matrix and $\sP$ be an elliptic system of differential operators with $\DC{\fM}{\Omega}$. We note that in that case
instead of $u\in\E(\Omega)$ we can just assume that $u\in\D^\prime(\Omega)$ in Definition \ref{VectorDef}
by the subellipticity of the elliptic system $\sP$,
cf.~\cite{MR1037999} or \cite{Smith1970}.\footnote{The same is true of course for the statements
	in Theorem \ref{Main-Weight} and Theorem \ref{MainTheoremOmega}.}
This allows us to deduce results on  ultradifferentiable hypoellipticity from
Theorem \ref{localThmIterates}.
\begin{Def}
Let $\fM$ be a weight matrix and $\sP=\{P_1,\dotsc,P_\ell\}$ be a system of  differential operators with $\DC{\fM}{\Omega}$-coefficients.
We say that $\sP$ is $[\fM]$-hypoelliptic if for any open $U\subseteq\Omega$ and all $u\in\Dp(U)$
the fact that $P_ju\in\DC{\fM}{U}$, $j=1,\dotsc,\ell$, implies that $u\in\DC{\fM}{U}$.
\end{Def}
\begin{Thm}\label{Regularity}
	Let $\fM$ be a [weakly regular] weight matrix and $\sP$ be an elliptic system of operator of class $[\fM]$
	in $\Omega$.
	Then $\sP$ is $[\fM]$-hypoelliptic in $\Omega$.
\end{Thm}

It is worthwile to compare the conditions on the weight matrix in 
Theorem \ref{Regularity} with the hypothesis needed in the microlocal regularity results given in \cite[Section 7]{FURDOS2020123451}.
For simplicity we restrict our discussion to  Denjoy-Carleman classes.
In \cite{FURDOS2020123451} we proved the following result:
\begin{Thm}[{\cite[Theorem 7.1 \& Theorem 7.4]{FURDOS2020123451}}]\label{MicrolocalThm}
	Let $\bM$ be a weight sequence satisfying \eqref{M0}, \eqref{M1} and \eqref{M2}.
	Then for any differential operator with coefficients in $\DC{\bM}{\Omega}$
	we have that
	\begin{equation*}
		\WF_{[\bM]}u\subseteq \WF_{[\bM]}Pu\cup\Char P
	\end{equation*}
for all $\Dp(\Omega)$.
\end{Thm}
Here $\WF_{[\bM]}u$ denotes the ultradifferentiable wavefront set with respect to the
weight sequence as defined by \cite{MR0294849} for Roumieu classes 
(for the Beurling case see \cite{FURDOS2020123451}).
Moreover, $\Char P$ is the characteristic set of the linear differential operator $P$,
cf.\ e.g.\ \cite{MR1996773}.
Since $\bigcap_{j=1}^\ell \Char P_j=\emptyset$ for any elliptic system of differential operators
and under our assumptations it holds that 
$\WF_{[\bM]}Pu\subseteq \WF_{[\bM]} u$ for all $u\in\Dp(\Omega)$ and any differential operator
$P$ of class $[\bM]$, cf.~\cite[Proposition 5.4(7)]{FURDOS2020123451}, we obtain the following corollary from Theorem \ref{MicrolocalThm}.
\begin{Cor}\label{microhypoelliptic}
	Let $\bM$ be a weight sequence satisfying \eqref{M0}, \eqref{M1} and \eqref{M2}
	and $\sP=\Set{P_1,\dotsc,P_\ell}$ be an elliptic system of differential operators 
	with coefficients in $\DC{\bM}{\Omega}$. Then 
	\begin{equation*}
		\WF_{[\bM]}u=\bigcap_{j=1}^\ell \WF_{[\bM]}P_ju
	\end{equation*}
for all $u\in\Dp (\Omega)$.
\end{Cor}
Corollary \ref{microhypoelliptic} gives that the system $\sP$ is $[\bM]$-hypoelliptic,
but the statement is in fact stronger, namely it says that the ultradifferentiable hypoellipticity
of $\sP$ holds on the microlocal level,
i.e.\ $\sP$ is $[\bM]$-microhypoelliptic. But as we have discussed in the introduction 
the assumptations on the weight sequence in Corollary \ref{microhypoelliptic}
are much stricter than the conditions in Theorem \ref{Regularity}.
In particular, Theorem \ref{Regularity} holds for the weight sequences $\bN^q$, $q>1$, given by $N_k^q=q^{k^2})_k$, but $\bN^q$ does not satisfy all the conditions in Corollary 
\ref{microhypoelliptic}.

\subsection{Global Kotake-Narasimhan Theorems}
Following \cite{MR548225} we can adapt the proof of Theorem \ref{localThmIterates} to 
obtain a global theorem of iterates.

Suppose that $\Omega\subseteq{\Real}^n$ is an open set with boundary $\partial\Omega$ 
and $\sP=\{P_1,\dotsc,P_\ell\}$ a system of differential operators defined 
on $\overline{\Omega}$ with principal symbols $p_j(x,\xi)$.
We say that the system $\sP$ is \emph{globally elliptic} in $\overline{\Omega}$ if
\begin{enumerate}
	\item $\sP$ is elliptic in the interior, i.e.\ $\Omega$.
	\item For all $x\in \partial X$ the polynomials $p_j(x,\xi)$ have no common 
	\emph{complex} zeros in $\xi\in{\Comp}^n\setminus\{0\}$. 
\end{enumerate}
We obtain
\begin{Thm}\label{globalThmIterates}
	Let $\fM$ be a [weakly regular] weight matrix,
	 $\Omega\subseteq{\Real}^n$ be a bounded open set with Lipschitzian boundary and
	$\sP=\{P_1,\dotsc,P_\ell\}$ be a globally elliptic system of partial differential operators
	with coefficients in $\gDC{\fM}{\Omega}$.
	Then
	\begin{equation*}
		\vgDC[\sP]{\fM}{\Omega}=\gDC{\fM}{\Omega}.
	\end{equation*}
\end{Thm}
The proof is based on a global a-priori estimate for globally elliptic systems $\sP$ of equal order
$d$: There is a constant $C>0$ such that for every $u\in\D(\overline{\Omega})$ and $k=1,\dotsc,d$
we have 
\begin{equation}\label{GlobalEstimate}
	\norm{u}{H^k(\Omega)}\leq C\left\{\sum_{j=1}^\ell \norm*{P_ju}{H^{k-d}(\Omega)}
	+\norm{u}{L^2(\Omega)}\right\},
\end{equation}
cf.\ \cite{Aronszajin}, \cite{Smith1970} and also \cite{Agmon1965}.
From this estimate, resp.\ \cite[Propositions I-2 \& I-3]{MR548225}, which are consequences of 
\eqref{GlobalEstimate} we can follow the lines of the proof of \cite[Theorem 2]{MR548225}.
We leave the details to the reader.

Similar to the local case Theorem \ref{globalThmIterates} yields results on the global ultradifferentiable hypoellipticity of globally elliptic system, cf.~\cite{MR548225}
for the Gevrey case.
We might also note, that we can directly generalize a characterization for global classes
given in \cite{MR548225}.
Let now $\sP=\{P_1,\dotsc,P_\ell\}$ be a system of differential operators $P_j=P_j(D)$
 with constant coefficients. We are going to assume that the system $\sP$ satisfies
 the following condition:
 \begin{equation}\label{ConstantsCOeff}
 	 \text{The set of common complex zeros of the polynomials } P_j(\xi),\,1\leq j\leq \ell,
 	 \text{ is finite}.
 \end{equation}
Here $P_j(\xi)$ denotes the \emph{full} symbol of the operator $P_j(D)$.
\begin{Thm}
	Let $\Omega\subseteq{\Real}^n$ be a bounded open set with Lipschitzian boundary, 
	$\sP=\{P_1,\dotsc,P_\ell\}$ be a system of operators with constant coefficients  and $\fM$ be a [weakly regular] weight matrix.
	Then the following statements are equivalent:
	\begin{enumerate}
	\item $\Set*{u\in\Dp(\Omega)\given P_ju\in\gDC{\fM}{\Omega},\;1\leq j\leq \ell}=\gDC{\fM}{\Omega}$
	\item The system $\sP$ satisfies \eqref{ConstantsCOeff}.
	\end{enumerate}
\end{Thm}
\begin{proof}
	First, assume that (1) holds. We set 
	\begin{equation*}
		Y(\Omega)=\Set*{u\in\Dp(\Omega)\given P_ju=0,\;1\leq j\leq \ell}\subseteq\gDC{\fM}{\Omega}
	\end{equation*}
and denote by $Y_k(\Omega)$ the space $Y(\Omega)$ equipped with 
the $H^k(\Omega)$-norm, for $k\in\N_0$. It is easy to see that the $Y_k(\Omega)$ are all Banach spaces. Moreover, the identity mapping
from $Y_{k+1}(\Omega)$ into $Y_k(\Omega)$ is clearly continuous
and therefore an isomorphism.
Thus all $H^k(\Omega)$-norms are pairwise equivalent to each other on $Y(\Omega)$.

On the other hand, $Y(\Omega)\subseteq \E(\overline{\Omega})$ is a closed subspace of
$\E(\overline{\Omega})$ with the usual topology.
The Sobolev embedding theorem implies moreover that the topology
 of $\E(\overline{\Omega})$
is generated by the system of seminorms $\Set{\norm{\,.\,}{H^k(\Omega)}\given k\in\N_0}$. 
Thus on $Y(\Omega)$ the topologies coming from $\E(\overline{\Omega})$ and $H^k(\Omega)$, $k\in\N_0$ agree.
Furthermore, $\Omega$ has only finitely many connected components $\Omega_j$ since $\Omega$ is bounded with Lipschitz boundary and for each two points $x,y$ in a connected components $\Omega_j$
there is a continuous path $\gamma$ connecting $x$ with $y$ such that the length of $\gamma$ is
smaller than $C_j\abs{x-y}$ where $C_j$ is a constant only depending on $\Omega_j$, see \cite{zbMATH06596721}.
It follows that $\E(\overline{\Omega})$ is nuclear and therefore also $Y(\Omega)$ according
to \cite{zbMATH01021609}.
Thence $Y(\Omega)$ is a nuclear Banach space 
and thus $Y(\Omega)$ has to be finite dimensional according to \cite{zbMATH05162168}. 

But if $\xi_0\in{\Comp}^n$ satisfies $P_j(\xi_0)=0$ for all $1\leq j\leq\ell$ then the function
$u(x)=e^{ix\xi_0}$ is a solution of $P_ju=0$ for all $1\leq j\leq \ell$.
Thence the set of all common complex zeros of the polynomials $P_j(\xi)$, $1\leq j\leq \ell$, 
has to be finite.

On the other hand, suppose that (2) is true. Let $\xi^1,\dots,\xi^\nu$ be the common complex
zeros of the polynomials $P_j(\xi)$, $1\leq j\leq \ell$.
For each $1\leq j\leq n$ we consider the polynomial 
\begin{equation*}
	Q_j(\xi)=\prod_{\kappa=1}^\nu \left(\xi_j-\xi_j^\kappa\right)
\end{equation*}
where $\xi=(\xi_1,\dotsc,\xi_n)$ and $\xi^\kappa=(\xi^\kappa_1,\dotsc,\xi^\kappa_n)$.
Then $Q_j(\xi^\kappa)=0$, $1\leq \kappa\leq\nu$, that means that the polynomials
$Q_j(\xi)$ vanish on the set of common complex zeros of the polynomials $P_j$, $1\leq j\leq \ell$.
The Nullstellensatz, cf.~\cite[Theorem IX.1.5]{MR1878556}, implies that there exists an integer
$\rho\geq 1$ such that the polynomials $Q_j^\rho$, $1\leq j\leq n$, belong to the ideal spanned
by the polynomials $P_r$, $1\leq r\leq \ell$. Thence, there exist polynomials $A_{jr}$ such that
\begin{equation*}
	Q_j^\rho(\xi)=\sum_{r=1}^\ell A_{jr}(\xi)P_r(\xi),\quad 1\leq j\leq n.
\end{equation*}
The polynomials $Q_j^\rho(\xi)$ are of order $\nu\rho$ whose principal part is equal 
to  $(\xi_j)^{\nu\rho}$. Thus $0$ is the only complex common zero of these principal parts
and therefore $\mathsf{Q}=\Set{Q_1^\rho,\dotsc,Q_n^\rho}$ is globally elliptic in $\overline{\Omega}$.
Furthermore, if $u\in\Dp(\Omega)$ and $P_ju\in\gDC{\fM}{\Omega}$ for $1\leq j\leq \ell$
then $Q_j^\rho(D)u\in\gDC{\fM}{\Omega}$ for $1\leq j\leq n$.
From \cite{MR0435573} it follows that $u\in\E(\overline{\Omega})$ and therefore
$u\in\vgDC[\mathsf{Q}]{\fM}{\Omega}$.
Finally, according to Theorem \ref{globalThmIterates} we have $u\in\gDC{\fM}{\Omega}$.
\end{proof}
\subsection{Final Remarks}
We can ask if the conditions, which we have imposed on the data of the ultradifferentiable class $\E^{[\fM]}$
for the Kotake-Narasimhan Theorem to hold, can be further loosened.
For this, we recall that we proved the Theorem of Iterates for $\E^{[\fM]}$ in the case
of elliptic operators with analytic coefficients when $\fM$ is [semiregular], cf.~\cite{zbMATH07538293}. But if $\fM$ is [semiregular] then $\E^{[\fM]}$ is closed under
derivation and invariant under composition with analytic mappings by \cite{FURDOS2020123451}.
On the other hand, if the weight matrix $\fM$ is [weakly regular], then $\E^{[\fM]}$ is closed under derivation and invariant under composition with ultradifferentiable mappings of class $[\fM]$.
Although we must note that the assumption of weakly regularity is not a priori optimal for this fact 
to hold, see \cite{MR3285413}.
However, in the case of Braun-Meise-Taylor classes we have that the space $\E^{[\omega]}$ is 
invariant under composition with maps of class $[\omega]$ if and only if $\omega$ is equivalent
to a concave weight function. We observe also that $\E^{[\omega]}$ is closed under derivation
by definition.

All these arguments motivate the following conjecture:
\begin{Conj}
	Let $\U$ be an ultradifferentiable structure which is closed under derivation
	and is invariant under composition with mappings of class $\U$. Then the Kotake-Narashiman Theorem holds in the class $\U$.
\end{Conj}

As we have stated, the conjecture is verified for Braun-Meise-Taylor classes, but we claim moreover
that the same is true for Denjoy-Carleman classes.
Recall from \cite{MR3285413} that a Denjoy-Carleman class $\DC{\bM}{\Omega}$ which
contains $\An(\Omega)$ (cf.~Remark \ref{StrongRemark}) and is closed under derivation,
i.e.~satisfies \eqref{M2prime},
is closed under composition with mappings of class $[\bM]$ if and only if the sequene
$\sqrt[k]{m_k}$ is \emph{almost increasing}, that is
\begin{equation}\label{almostIncr}
	\ex C>0:\quad \sqrt[j]{m_j}\leq C\sqrt[k]{m_k}\qquad \fa j\leq k.
\end{equation}

In order to prove our claim, let $\bM$ be a weight sequence such that  \eqref{almostIncr} and,
if we exclude the analytic case, \eqref{AnalyticInclusion1} hold.
We define a new sequence $\widetilde{\bM}$ by the following procedure.\footnote{It is inspired by \cite[Lemma 8]{RainerSchindl2020}.}
For $k\in\N$ we set
\begin{equation*}
	\nu_k=C\inf_{\ell\geq k}\sqrt[\ell]{m_\ell}
\end{equation*}
where $C>0$ is the constant from \eqref{almostIncr}.
Thence the sequence $\nu_k$ is increasing and $\sqrt[k]{m_k}\leq\nu_k\leq C\sqrt[k]{m_k}$.
 We define $\widetilde{\bM}$ by
$\widetilde{M}_0=1$ and $\widetilde{M}_k=k!\nu_k^k$ for $k\in\N$, in particular $\widetilde{M}_1\geq 1$.
It follows that $\widetilde{\bM}$ satisfies \eqref{AnalyticInclusion1} and \eqref{RootIncreasing}.
We need to point out, that we cannot conclude that $\widetilde{\bM}$ satisfies \eqref{WeakLogConvex} and therefore cannot assume that $\widetilde{\bM}$ is a weigth sequence.
But a close inspection of the proof of Theorem \ref{localThmIterates} in both the Roumieu and Beurling case shows that we still obtain the assertion of Theorem \ref{localThmIterates}
for the sequence $\widetilde{\bM}$. 
 But since $\widetilde{\bM}\approx\bM$, i.e. $\DC{\bM}{\Omega}=\DC{\widetilde{\bM}}{\Omega}$
 and $\vDC{\bM}{\Omega}=\vDC{\widetilde{\bM}}{\Omega}$, we have in fact shown the assertion of Theorem 
 \ref{localThmIterates} for $\bM$. Therefore the conjecture is also true for Denjoy-Carleman classes.

 \begin{Rem}
 	In view of the proof of
 	Proposition \ref{KomatsuTrick} and the argument above, it would make sense to adapt
 	the definition of a weight sequence by replacing \eqref{WeakLogConvex} by the following condition:
 	\begin{equation*}
 		\text{the sequence }\sqrt[k]{M_k}\text{ is increasing.}
 	\end{equation*}
 However, \eqref{WeakLogConvex} is a standard assumption for weight sequences in context
 of Denjoy-Carleman classes, see e.g.~\cite{Komatsu73} or \cite{MR3285413} for classes given
 by weight matrices.
 Moreover, we have needed the concept of logarithmic convexity for the proof of Proposition
 \ref{KomatsuTrick}.
 \end{Rem}

\section*{Acknowledgments}
The author has been supported by Austrian Science Fund (FWF) grant J4439
while he has been a post doctoral researcher under the supervision of Paulo D.~Cordaro at USP-IME from October 2021 to September 2023.\\
The author would also like to thank Gerhard Schindl for numerous helpful discussions.
\bibliographystyle{abbrv}
\bibliography{vectors2}
\end{document}